\documentclass[12pt,reqno]{amsart}

\usepackage{amssymb,latexsym}

\usepackage{enumerate}

\usepackage[french,english]{babel}
\usepackage{amsmath}
\usepackage{graphicx}
\usepackage{amssymb}
\usepackage{bbm}
\usepackage{amsthm,mathtools}
\usepackage{ulem}
\usepackage{geometry}
\usepackage{tikz-cd}
\usepackage{mathrsfs}
\usepackage[colorinlistoftodos]{todonotes}
\usepackage{enumitem}
\usepackage{verbatim}
\usepackage[foot]{amsaddr}
\usepackage{dsfont}

\makeatletter

\@namedef{subjclassname@2010}{
	
	\textup{2020} Mathematics Subject Classification}

\makeatother
\newtheorem{thm}{Theorem}[section]
\newtheorem*{thm*}{Theorem}

\newtheorem{lem}[thm]{Lemma}

\theoremstyle{definition}

\numberwithin{equation}{section}

\newcommand{\mcm}{\mathcal{M}}

\newcommand{\mmd}{\mathrm{d}}
\newcommand{\mme}{\mathrm{e}}

\newcommand{\M}{{\mathcal M}}

\newcommand{\F}{\mathcal{F}}

\usepackage{hyperref}
\hypersetup{hypertex=true,colorlinks=true,linkcolor=blue,anchorcolor=blue,citecolor=blue}
\newcommand{\newabstract}[1]{%
	\par\bigskip
	\csname otherlanguage*\endcsname{#1}%
	\csname captions#1\endcsname
	\item[\hskip\labelsep\scshape\abstractname.]
}

\begin{document}

	\baselineskip=17pt

	\title[Extreme values of quadratic Dirichlet $L$-functions]{Extreme values of quadratic Dirichlet $L$-functions}

	\author{Zikang Dong}
    \author{Weijia Wang}
    \author{Hao Zhang}
	
    \author{Shengbo Zhao}
	\address[Zikang Dong]{School of Mathematical Sciences, Soochow University, Suzhou 215006, P. R. China}
     \address[Weijia Wang]{School of Mathematics, Shandong University, Jinan 250100, P. R. China}
	\address[Hao Zhang]{School of Mathematics, Hunan University, Changsha 410082, P. R. China}
	
    \address[Shengbo Zhao]{School of Mathematical Sciences, Key Laboratory of Intelligent Computing and Applications (Tongji University), Ministry of Education, Tongji University, Shanghai 200092, China}
	\email{zikangdong@gmail.com}
    \email{weijiawang@amss.ac.cn}
	\email{zhanghaomath@hnu.edu.cn}
	
    \email{shengbozhao@hotmail.com}
	
	\date{}
	
	\begin{abstract} 
		In this article, we investigate extreme values of quadratic Dirichlet $L$-functions at the central point. We provide  new extreme values of $L(\frac12,\chi_d)$ as $d$ is large, which improves the recent result of Darbar and Maiti.
	\end{abstract}

	
	\maketitle
	
\section{Introduction}
Throughout this paper, we write $\log_j$ for the $j$-th iterated logarithm, such as \(\log_2x= \log\log x\), and \(\log_3 x=\log\log\log x\). In recent years, research on extreme values has advanced significantly, due to the resonance method introduced by Hilberdink \cite{HIL} and developed by Soundararajan \cite{Sound}. In \cite{Sound}, Soundararajan showed extreme values for the Riemann zeta function, the quadratic  Dirichlet $L$-functions, and the $L$-functions for the cusp forms. For the Riemann zeta function, he showed for large $T$
$$\max_{\substack{T<|t|\le2T}}\big|\zeta(\tfrac12+it)|\ge\exp\bigg(\big(1+o(1)\big)\sqrt{\frac{\log T}{\log_2T}}\bigg).$$
 This was improved in 2017 by Bondarenko and Seip \cite{BS2}
$$\max_{\substack{0<|t|\le T}}\big|\zeta(\tfrac12+it)|\ge\exp\bigg(\big(\tfrac1{\sqrt2}+o(1)\big)\sqrt{\frac{\log T\log_3T}{\log_2T}}\bigg).$$
 This breakthrough was based on the connection between the Riemann zeta function and GCD sums, which was first observed by Aistleitner \cite{Ais}. Bondarenko and Seip \cite{BS} also improved  the constant $\frac 1{\sqrt 2}$ to 1 in 2018. After then, in 2019 La Bret\`eche and Tenenbaum \cite{BT} improved this to $\sqrt2$ and this is the best known result till now.

For the Dirichlet $L$-functions of prime modulus $q$, 
La Bret\`eche and Tenenbaum \cite{BT} showed that, for sufficiently large \(q\),
$$ 
\max_{\substack{\chi\neq\chi_0({\rm mod}\; q) \\ \chi(-1)=1}}\big|L(\tfrac12,\chi)|\ge\exp\bigg((1+o(1))\sqrt{\frac{\log q\log_3q}{\log_2q}}\bigg)
$$
Now let  \(\F\) denote the set of all fundamental discriminants
and \(\chi_d := \big(\frac{d}{\cdot} \big)\) be the real primitive character modulo \(|d|\). For the quadratic Dirichlet $L$-functions $L(\frac12,\chi_d)$, Soundararajan \cite[Theorem 2]{Sound} showed that for large $X$
$$ 
\max_{\substack{X<|d|\le2X \\ d \in \F}}\big|L(\tfrac12,\chi_d)|\ge\exp\bigg(\big(\tfrac1{\sqrt5}+o(1)\big)\sqrt{\frac{\log X}{\log_2X}}\bigg).
$$

This was not improved until recently Darbar and Maiti \cite{DM} showed the following, under the assumption of the Generalized Riemann Hypothesis (GRH), 
   $$
   \max_{\substack{X<|d|\le2X \\ d \in \F}}\big|L(\tfrac12,\chi_d)|\ge\exp\bigg((\tfrac{1}{2}+o(1))\sqrt{\frac{\log X\log_3X}{\log_2X}}\bigg).
   $$
The aim of this paper is to improve the constant $1/2$ to 1. More precisely, We have the following theorem.
\begin{thm}
\label{thm1} 
Assuming GRH, we have for sufficiently large $X$
       $$ \max_{\substack{X<|d|\le2X \\ d \in \F}}\big|L(\tfrac12,\chi_d)|\ge\exp\bigg((1+o(1))\sqrt{\frac{\log X\log_3X}{\log_2X}}\bigg).$$
\end{thm}

    \section{Preliminary Lemmas}
    \label{sec1.2}
    In this section, we present several lemmas that will be used in the subsequent proofs. We begin with an approximate functional equation for quadratic Dirichlet \(L\)-functions.
    \begin{lem}
    \label{appox}
    Let $\chi_d$ be a quadratic Dirichlet character. Then
    $$
    L(\tfrac12,\chi_d) = 2\sum_{n \ge 1}\frac{\chi_d(n)}{\sqrt{n}}\omega\Big( \frac{n}{\sqrt{d}}\Big),
    $$
where
$$
\omega(\xi):=\frac{1}{2\pi i}\int_{(c)} \pi^{-\frac{s}{2}}\frac{\Gamma(\frac s2+\frac14)}{\Gamma(\frac14)}\xi^{-s}\frac{\mmd s}{s}
$$
with $c>0$. The function $\omega(x)$ is real-valued, smooth on $(0,+\infty)$. Moreover, we have $\omega(\xi)=1+O(\xi^{1/2-\varepsilon})$ as $\xi \to 0$, and $\omega(\xi) \ll \mme^{-\xi}$ as $\xi \to \infty$. 
    \end{lem}
\begin{proof}
    This is \cite[Lemmas 2.1 and 2.2]{Sound0}, also see \cite[Lemma 7]{DM}.
\end{proof}

The following result gives conditional estimates for the mean values of quadratic characters. This improves the earlier unconditional results due to \cite[Lemma 4.1]{GS}.
    \begin{lem}
    \label{charsum}
	Assume GRH. Let $n=n_1 n_2^2$ be a positive integer with $n_1$ the square-free part of $n$. Then for any $\varepsilon>0$, we obtain
	\begin{align*}
	\sum_{|d|\le X\atop d\in\F} \chi_{d}(n)=\frac{X}{\zeta(2)}\prod_{p|n}\Big(\frac{p}{p+1}\Big) \mathds{1}_{n=\square}+ O\big(X^{1/2+\varepsilon}g_1(n_1)g_2(n_2)\big),
	\end{align*}
	where  $g_1(n_1)=\exp\left((\log n_1)^{1-\varepsilon}\right)$ arises from the square-free component, while $g_2(n_2)=\sum_{q\mid n_2}\frac{\mu^2(q)}{q^{1/2+\varepsilon}}$ originates from the square component of the modulus, and ${\mathds{1}}_{n=\square}$ indicates the indicator function of the square numbers. 
\end{lem}
\begin{proof}
    This is \cite[Lemma 2]{DM}.
\end{proof}
It is clear for $g_1(n_0)$ that 
$$g_1(n_0)\le  \exp\big((\log n)^{1-\varepsilon}\big).$$
For $g_2(n_1)$ we have
$$g_2(n_1)=\prod_{p|n_1}\Big(1-\frac1{p^{1/2+\varepsilon}}\Big)\le\exp\big((\log n_1)^{\frac12-\varepsilon}\big)\le\exp\big((\log n)^{\frac12-\varepsilon}\big).$$
Moreover, let $P_+(n)$ denote the largest prime divisor of $n$, then we also have
$$g_1(n_0)\le  \exp\big((P_+(n))^{1-\varepsilon}\big),\;\;\;\;\;\;\;g_2(n_1)\le\exp\big((P_+(n))^{\frac12-\varepsilon}\big).$$

Let \((m,n)\) and \([m,n]\) denote the greatest common divisor and the least common multiple of positive integers \(m\) and \(n\), respectively. The following result for GCD sums plays a key role in the proof of Theorem \ref{thm1}.
\begin{lem}\label{GCD}
    Let $\M$ be any set of positive square-free integers with $|\M|=N$. Then as $N\to\infty$, we have
   $$\max_{|\M|=N}\sum_{m,n\in\M}\sqrt{\frac{(m,n)}{[m,n]}}=N\exp\bigg((2+o(1))\sqrt{\frac{\log N\log_3N}{\log_2N}}\bigg).$$
\end{lem}
\begin{proof}
    This is \cite[Eq. (1.5)]{BT}.
\end{proof}
Note that in the proof, the choice for the set $\M$ satisfies that
$$y_\M :=\max_{m\in\M}P_+(m)\le (\log N)^{1+o(1)}.$$ 

   \section{Proof of Theorem \ref{thm1}}

   Let \(X\) be large, and let \(N=X^{1/4-\alpha}\) for small \(\alpha>0\). Define the resonator \(R_d = \sum_{n \in \mcm} \chi_d(n)\), where \(d \in \F\), and the set \(\M\) satisfies the condition of Lemma \ref{GCD}. Let \(P_+(n)\) denote the largest prime divisor of \(n\). Note that the construction of $\M$ satisfies $$y_\M :=\max_{m\in\M}P_+(m)\le (\log N)^{1+o(1)}.$$ Then, we define the following two sums:
    \begin{align*}
        S_1 &: = S_1(R_d,X) =\sum_{\substack{X<|d|\le2X \\ d \in \F}}R_d^2 , \\
        S_2 &: = S_2(R_d,X) =\sum_{\substack{X<|d|\le2X \\ d \in \F}}L(\tfrac12,\chi_d) R_d^2.
    \end{align*}

Trivially, we have
\begin{align}
    \label{max}
    \max_{\substack{X<|d|\le2X \\ d \in \F}}L(\tfrac12,\chi_d) \ge \frac{S_2}{S_1}.
\end{align}
We next find an effective lower bound for the ratio \(S_2/S_1\). For \(S_1\), expanding \(R_d^2\) yields that
\begin{align*}
    S_1 &= \sum_{\substack{X<|d|\le2X \\ d \in \F}}R_d^2 =\sum_{m,n \in \mcm}\sum_{\substack{X<|d|\le2X \\ d \in \F}}\chi_d(m)\chi_d(n) \\
    & = \sum_{m,n \in \mcm}\Big( \sum_{\substack{X<|d|\le2X \\ d \in \F, mn=\square}}+\sum_{\substack{X<|d|\le2X \\ d \in \F,mn\neq\square}}\Big)\chi_d(mn).
\end{align*}

Applying Lemma \ref{charsum}, we obtain
\begin{align*}
    S_1 = \frac{X}{\zeta(2)}\sum_{\substack{m,n \in \mcm \\ mn=\square}} \prod_{p \mid mn}\Big(\frac{p}{p+1}\Big) + O\Big(X^{1/2+\varepsilon}\sum_{\substack{m,n\ge 1 \\ mn=n_1n_2^2,\mu(n_1)\neq0}}g_1(n_1)g_2(n_2) \Big).
\end{align*}
Let \(h(1)=1\) and 
\[
h(n) = \prod_{p \mid n}\Big( \frac{p}{p+1} \Big)
\]
for \(n \ge 2\). It is easy to see that \(0<h(n)\le 1\) for all \(n \ge 1\). Combining this with the upper bounds for $g_1$ and $g_2$ stated after Lemma \ref{charsum}, we derive that
\[
     S_1 = \frac{X}{\zeta(2)}\sum_{\substack{m,n \in \mcm \\ mn=\square}}h(mn) + O\Big(X^{1/2+\varepsilon}\sum_{m,n\in \M}1\Big).
\]
Since \(\M\) is a set of square-free integers, for any \(m,n\in\M\), \(mn=\square\) implies \(m=n\). Combining this with \(|\M|=N\), we obtain
\begin{align*}
     S_1 = \frac{X}{\zeta(2)} \sum_{m \in \M}h(m) + O\big( X^{1/2+\varepsilon} N^{2}\big) \le \frac{X}{\zeta(2)}N+O\big( X^{1/2+\varepsilon} N^{2}\big) .
\end{align*}
It follows from \(N=X^{1/4-\alpha}\) that the error is negligible, and
\begin{equation}
    \label{S1upper}
    S_1 \le \frac{X}{\zeta(2)}N+O(X^{1-2\alpha+\varepsilon}).
\end{equation}

On the other hand, for \(S_2\), using Lemma \ref{appox} and expanding \(R_d^2\) again, we have
\begin{align*}
    S_2 &= 2 \sum_{m,n\in\M}\sum_{k \ge 1}\frac{1}{\sqrt{k}} \sum_{\substack{X<|d|\le2X \\ d \in \F}} \chi_d(kmn)\omega\Big(\frac{k}{\sqrt{d}}\Big) \\
    & = 2 \sum_{m,n \in \mcm}  \sum_{k \ge 1}\frac{1}{\sqrt{k}}\Big( \sum_{\substack{X<|d|\le2X \\ d \in \F, kmn=\square}}+\sum_{\substack{X<|d|\le2X \\ d \in \F,kmn\neq\square}}\Big)\chi_d(kmn)\omega\Big(\frac{k}{\sqrt{d}}\Big).
\end{align*}
Applying Lemma \ref{charsum} once more to the innermost sum in the above expression, we deduce that
\begin{align}
\label{S2eq}
    S_2 &= \frac{2X}{\zeta(2)}\sum_{m,n\in\M}\sum_{\substack{k \ge 1\\kmn=\square}}\frac{1}{\sqrt{k}}h(kmn)\int_1^2\omega\Big( \frac{k}{\sqrt{Xu}}\Big) \mmd u \nonumber\\
    &\quad + O\Big(X^{1/2+\varepsilon}\sum_{k \ge 1}\frac{1}{\sqrt{k}} \mme^{-k/\sqrt{X}}\sum_{\substack{m,n\ge 1 \\ mn=n_1n_2^2,\mu(n_1)\neq0}}g_1(n_1)g_2(n_2) \Big) \nonumber\\
    & =: D+O(E).
\end{align}
By an argument similar to \cite[pp. 13-14]{DM}, the error of \(S_2\) satisfies
\begin{equation}
    \label{error}
    E \ll X^{3/4+\varepsilon}N^\varepsilon|\M|^2 \ll X^{3/4+\varepsilon}N^{2+\varepsilon}.
\end{equation}
For the main term \(D\) of \(S_2\), since both \(h(n)\) and \(\omega(\xi)\) are non-negative, we may retain only the term with 
\[
k = \frac{[m,n]}{(m,n)},
\]
which yields 
\begin{align}
\label{Dlower}
    D & \ge \frac{2X}{\zeta(2)} \sum_{m,n \in \M}h\Big(\frac{[m,n]mn}{(m,n)}\Big)\sqrt{\frac{(m,n)}{[m,n]}} \int_1^2\omega\Big(\frac{[m,n]}{(m,n)\sqrt{Xu}} \Big)\mmd u \nonumber \\
    & \ge \frac{2X}{\zeta(2)} \sum_{\substack{m,n\in\M \\ [m,n]/(m,n) \le X^\varepsilon}} h(mn)\sqrt{\frac{(m,n)}{[m,n]}} \big(1+O\big((X^{-1/2+\varepsilon})^{1/2-\varepsilon}\big)\big).
\end{align}
Here, in the last step, we impose a restriction on \([m,n]/(m,n) \le X^\varepsilon\) in order to employ the asymptotic formula for \(\omega(\xi)\) when \(\xi\) is small from Lemma \ref{appox}.

By the definition of \(\M\) and \(h(n)\), we have
\[
D \gg \frac{X}{\zeta(2)} \prod_{p \le (\log N)^{1+o(1)}} \Big(\frac{p}{p+1}\Big) \sum_{\substack{m,n\in\M \\ [m,n]/(m,n) \le X^\varepsilon}} \sqrt{\frac{(m,n)}{[m,n]}} .
\]
Note that
\begin{align*}
 \prod_{p \le (\log N)^{1+o(1)}} \Big(\frac{p}{p+1}\Big) &=\exp\Big(\sum_{p \le (\log N)^{1+o(1)}} \log \Big( 1-\frac{1}{p+1} \Big)\Big)   \\
 & \ge \exp \Big(-\sum_{p \le (\log N)^{1+o(1)}} \frac{1}{p}\Big).
\end{align*} 
Thus, the prime number theorem implies 
\begin{equation}
    \label{prodlower}
    \prod_{p \le (\log N)^{1+o(1)}} \Big(\frac{p}{p+1}\Big) \ge \exp(-\log_3 N).
\end{equation}
We now handle the remaining sum. Following the argument in \cite[p. 25]{BT}, for fixed \(m \in \M\), we have
\[
\sum_{n \in \M}\Big( \frac{(m,n)}{[m,n]} \Big)^{1/3} \le \prod_{p \le y_\M}\Big(1+\frac{2}{p^{1/3}-1}\Big) \ll \exp \big(y_\M^{2/3}\big).
\]
Furthermore, Rankin's trick shows that
\begin{align*}
    \sum_{\substack{m,n\in\M \\ [m,n]/(m,n) \le X^\varepsilon}} \sqrt{\frac{(m,n)}{[m,n]}} &= \Big(\sum_{m,n\in\M} - \sum_{\substack{m,n\in\M \\ [m,n]/(m,n) > X^\varepsilon}}\Big) \sqrt{\frac{(m,n)}{[m,n]}} \nonumber \\
    & \ge \sum_{m,n\in\M}  \sqrt{\frac{(m,n)}{[m,n]}} -X^{\varepsilon/6}\sum_{m,n\in\M} \Big(\frac{(m,n)}{[m,n]}\Big)^{1/3}\nonumber \\
    & \gg\sum_{m,n\in\M}  \sqrt{\frac{(m,n)}{[m,n]}} -X^{\varepsilon/6}|\M| \exp \big(y_\M^{2/3}\big).
\end{align*}
Since \(y_\M \le  (\log N)^{1+o(1)}\), applying Lemma \ref{GCD}, we have
\begin{align}
    \label{gcdlower}
 \sum_{\substack{m,n\in\M \\ [m,n]/(m,n) \le X^\varepsilon}} \sqrt{\frac{(m,n)}{[m,n]}} \ge N  \exp\bigg((2+o(1))\sqrt{\frac{\log N\log_3N}{\log_2N}}\bigg).
\end{align}
Combining \eqref{gcdlower} with \eqref{Dlower} and \eqref{prodlower}, we obtain the following lower bound for \(D\) as \(X\) is large:
\begin{equation}
    \label{Dlower2}
    D \ge \big(1+o(1)\big) \frac{X}{\zeta(2)} N \exp\bigg((2+o(1))\sqrt{\frac{\log N\log_3N}{\log_2N}}\bigg).
\end{equation}

It follows from \eqref{S1upper}, \eqref{S2eq}, \eqref{error} and \eqref{Dlower2} that
\[
\frac{S_2}{S_1} \ge \exp\bigg((2+o(1))\sqrt{\frac{\log N\log_3N}{\log_2N}}\bigg).
\]
Substituting this into \eqref{max} yields that
\begin{align*}
     \max_{\substack{X<|d|\le2X \\ d \in \F}}L(\tfrac12,\chi_d) & \ge   \exp\bigg((2+o(1))\sqrt{\frac{\log N\log_3N}{\log_2N}}\bigg) \\
     & \ge  \exp\bigg(\bigg(2\sqrt{\frac{1}{4}-\alpha}+o(1)\bigg)\sqrt{\frac{\log X\log_3X}{\log_2X}}\bigg).
\end{align*}
Taking \(\alpha \to 0^+\), we complete the proof of Theorem \ref{thm1}.


	\section*{Acknowledgements}
	Z. Dong is supported by the National
	Natural Science Foundation of China (Grant No. 	1240011770). W. Wang is supported by the National
	Natural Science Foundation of China (Grant No. 1250012812). H. Zhang is supported by the Fundamental Research Funds for the Central Universities (Grant No. 531118010622), the National
	Natural Science Foundation of China (Grant No. 1240011979) and the Hunan Provincial Natural Science Foundation of China (Grant No. 2024JJ6120).

	\normalem

\end{document}